\documentclass[11pt]{amsart}
\usepackage{amssymb}%para las letras dobles

\newcommand{\rnm}{{\mathbb{R}^n_+}}
\newcommand{\rn}{{\mathbb{R}^n}}
\newcommand{\phii}{\varphi}

\newcommand{\cinfc}{C^\infty_c}
\usepackage{color}

\def\R{{\mathbb {R}}}
\def\N{{\mathbb {N}}}
\def\e{{\mathbf {e}}}

\def\diam{\operatorname {\mathrm{diam}}}
\def\supp{\operatorname {\mathrm{supp}}}

\newtheorem{teo}{Theorem}[section]
\newtheorem{lema}[teo]{Lemma}

\theoremstyle{remark}
\newtheorem{remark}[teo]{Remark}

\theoremstyle{definition}

\numberwithin{equation}{section}

\title[Sobolev inequalities in variable exponent spaces]{A mass transportation approach for Sobolev inequalities in variable exponent spaces}

\author[J.P. Borthagaray, J. Fern\'andez Bonder and A. Silva]{Juan Pablo Borthagaray, Juli\'an Fern\'andez Bonder and Anal\'{\i}a Silva}

\address[J.P. Borthagaray and J. Fern\'andez Bonder]{IMAS - CONICET and Departamento de Matem\'atica, FCEyN - Universidad de Buenos Aires, Ciudad Universitaria, Pabell\'on I  (1428) Buenos Aires, Argentina.}

\address[A. Silva]{IMASL - CONICET and Departamento de Matem\'atica, Universidad Nacional de San Luis (5700) San Luis, Argentina}

\email[J.P. Borthagaray]{jpbortha@dm.uba.ar}

\email[A. Silva]{acsilva@unsl.edu.ar}

\email[J. Fern\'andez Bonder]{jfbonder@dm.uba.ar}

\urladdr[J. Fern\'andez Bonder]{http://mate.dm.uba.ar/~jfbonder}



\begin{document}

\begin{abstract}
In this paper we provide a proof of the Sobolev-Poincar\'e inequality for variable exponent spaces by means of mass transportation methods, in the spirit of \cite{CE-N-V}. The importance of this approach is that the method is flexible enough to deal with different inequalities. As an application, we also deduce the Sobolev-trace inequality improving the result of \cite{Fan} by obtaining an explicit dependence of the exponent in the constant.
\end{abstract}

\maketitle

\section{Introduction}

The goal of this paper is to show how mass transportation techniques can be applied to prove Sobolev inequalities in the context of variable exponent spaces.

Mass transportation is a subject that originates in the work of Monge in the XVIII century (cf. \cite{Monge}) and was mathematically stated in modern terms in the work of Kantorovich in the 1940s, \cite{Kantorovich}.

This topic has experienced a revolution since the by now classical paper of Brenier in 1987 (see \cite{Brenier1, Brenier2}). It is by now almost impossible to give a complete list of references or even topics where mass transportation methods are applied. We refer to the excellent books of Villani \cite{Villani1, Villani2}.

The application of mass transportation methods to Sobolev inequalities was first made by Cordero-Erausquin, Nazaret and Villani in \cite{CE-N-V}. See also \cite{Nazaret} where the trace inequality was studied.

The first inequality that we are going to deal with is the {\em Sobolev-Poincar\'e inequality}.

Given a measurable function $p\colon \R^n\to \R$ such that
\begin{equation}\label{p-p+}
1\le p_-:=\inf p \le p_+ := \sup p < n
\end{equation}
(here and throughout the paper, by inf and sup we mean the essential infimum and the essential supremum respectively), the Sobolev-Poincar\'e inequality states the existence of a constant $C>0$ such that
\begin{equation}\label{Sobolev-Poincare}
\|f\|_{p^*} \le C \|\nabla f\|_p
\end{equation}
for every $f\in C^\infty_c(\R^n)$, where $p^* = \frac{np}{n-p}$ and the norms are the so-called Luxemburg norms associated to the modular
$$
\rho_{r}(f):=\int_{\R^n} |f|^{r}\, dx.
$$
That is,
$$
\|f\|_{p^*} := \inf\left\{\lambda>0\colon \rho_{p^*}\left(\frac{f}{\lambda}\right)\le 1\right\}, \qquad \|\nabla f\|_{p} := \inf\left\{\lambda>0\colon \rho_{p}\left(\frac{|\nabla f|}{\lambda}\right)\le 1\right\}
$$

The validity of \eqref{Sobolev-Poincare} for constant exponents $p$ is well-known and we refer to the book of Adams, \cite{Adams}.

For variable exponents, the validity of \eqref{Sobolev-Poincare} was established in \cite{Diening, Edmunds-R1, Edmunds-R2, H-H}. See \cite[Theorem 8.3.1]{libro} for a proof.

The hypotheses on $p$ for \eqref{Sobolev-Poincare} to hold are, in addition to \eqref{p-p+}, that $p$ be globally log-H\"older continuous. See \cite[Chapter 4]{libro} for the definition of log-H\"older continuity.

Moreover, in \cite[Theorem 8.3.1]{libro}, it is shown that the constant $C$ in \eqref{Sobolev-Poincare} depends only con $n, p_+$ and the log-H\"older constant of $p$ denoted by $c_{\log}(p)$.

The proof in \cite{libro} is based in Harmonic Analysis techniques. More precisely, they use the boundedness of the Hardy--Littlewood maximal function and of the Riesz potentials in variable exponent spaces.

The mass transportation approach to this problem, is more direct and elementary. The only technical result that is needed is Brenier's theorem that asserts the existence of a transport $T$ between two probability measures and that this transport is the gradient of a convex function. This approach is the same that was used by Cordero-Erausquin, Nazaret and Villani in \cite{CE-N-V} dealing with the constant exponent case. See next section for the details.

Even though this method provides a more elementary proof of the Sobolev--Poincar\'e inequality \eqref{Sobolev-Poincare}, the main drawback is that we end up with a more restrictive hypotheses on the exponent $p$. We require that $p$ be differentiable. More precisely, we need to ask for $p\in W^{1,s}(\R^n)$, for some $s>n$.

Nevertheless, since the proof simpler, it is flexible enough to deal with other inequalities. As an example, in Section 3, we treat the Sobolev trace inequality,
\begin{equation}\label{trace}
\|f\|_{p_*, \R^{n-1}}\le C\|\nabla f\|_{p, \R^n_+}
\end{equation}
for every $f\in C^\infty_c(\R^n)$ with $C>0$ independent of $f$, where $p_* = \frac{(n-1)p}{n-p}$ and the (Luxemburg) norms are defined analogously as before.

The proof of this inequality in the constant exponent case is classical (see again the book of Adams \cite{Adams}). Also recall that mass transportation methods were applied in the constant exponent case for the Sobolev trace inequality \eqref{trace} by Nazaret in \cite{Nazaret}. Our proof follows closely the one in \cite{Nazaret}.

The variable exponent case of \eqref{trace} was proved by Fan in \cite{Fan}. Here we recover Fan's result and, moreover, by our method we can give the precise dependence of $C$ on the regularity of the exponent $p$ that was missing in \cite{Fan}.

\subsection*{Organization of the paper} After this short introduction, the paper is divided into two sections. Section 2 deals with the Sobolev-Poincar\'e inequality \eqref{Sobolev-Poincare} and Section 3 deals with the Sobolev trace inequality \eqref{trace}.

\section{The Sobolev-Poincar\'e inequality}

As we mentioned in the introduction, the mass transportation approach to the proof of the Sobolev-Poincar\'e inequality \eqref{Sobolev-Poincare} follows the lines of the paper by Cordero-Erausquin, Nazaret, Villani \cite{CE-N-V}.  We will point out the differences in the arguments when they arrive.

First, observe that is enough to prove \eqref{Sobolev-Poincare} for nonnegative functions $f\in C^\infty_c(\R^n)$. So, let $f, g \in \cinfc(\rn)$ be nonnegative functions such that $\|f\|_{p^*} = \|g\|_{p^*}=1$ and define the probability densities $F = f^{p^*}$ and $G=g^{p^*}$.

According to Brenier's theorem, \cite{Brenier1, Brenier2} (see \cite[Theorem 3.8]{Villani1} for a proof), there exists a convex lower semi-continuous function $\phii$ such that $\nabla \phii$
transports the measure $d\mu=F\, dx$ (optimally) to $d\nu = G\, dx$.

This implies, in particular, that for all $\psi \in L^1(d\nu)$ the following transport identity is valid:
\begin{equation}\label{transport.identity}
\int_\rn (\psi\circ \nabla \phii)\, d\mu = \int_\rn \psi\, d\nu.
\end{equation}
Moreover, the following Monge-Amp\`ere equation
\begin{equation}\label{Monge-Ampere}
F(x) = G (\nabla \phii(x)) \det(D_A^2\phii(x))
\end{equation}
holds $\mu-$a.e., where $D_A^2\phii$ is the Alexandrov Hessian of $\phii$. See \cite{McCann}.

Let us recall that the Alexandrov Hessian of a convex function $\phii$ is the absolute continuous part of the distributional Hessian of $\phii$ and that by Alexandrov's Theorem, it coincides with the Hessian of $\phii$ a.e. with respect to Lebesgue's measure. See \cite[Chapter 6]{Evans-Gariepy}.

Let us define the exponent
$$
p_*=\frac{(n-1)p}{n-p}=\left(1-\frac{1}{n}\right)p^*.
$$
Therefore, by \eqref{transport.identity} and \eqref{Monge-Ampere},
$$
\int_\rn g^{p_*}\, dx = \int_\rn G^{1-\frac{1}{n}}\, dx = \int_\rn (G^{-\frac{1}{n}}\circ \nabla \phii) \, d\mu = \int_\rn F^{1-\frac{1}{n}} \det \left(D^2_A \phii \right)^\frac{1}{n}\, dx.
$$

Since the distributional Hessian of $\phii$ is nonnegative (recall that $\phii$ is convex), it holds that $D^2_A\phii \leq D^2 \phii$ in the distributional sense. Thus, by applying first the arithmetic-geometric inequality for nonnegative symmetric matrices and integrating by parts we obtain
$$
\int_\rn g^{p_*}\, dx \leq \frac{1}{n} \int_\rn F^{1-\frac{1}{n}} \Delta_A \phii\, dx \le \frac{1}{n} \int_\rn F^{1-\frac{1}{n}} \Delta\phii\, dx =
- \frac{1}{n} \int_\rn \nabla \left(F^{1-\frac{1}{n}}\right) \cdot \nabla \phii \, dx.
$$
So far, the argument has been exactly as in Cordero-Erausquin, Nazaret, Villani's paper \cite{CE-N-V} where the reader can check the missing details in the arguments.

Now is where the difference appear since $p$ is not constant. From now on, assume that $p$ is smooth and that the bounds \eqref{p-p+} hold.

A straightforward computation gives
$$
\nabla \left(F^{1-\frac{1}{n}}\right) = \frac{(n-1) n f^{p_*} \log f}{(n-p)^2} \nabla p + p_* f^{p_*-1} \nabla f.
$$
 So we obtain the estimate
 \begin{align*}
\int_\rn g^{p_*}\, dx \leq \frac{(n-1)}{n-p_+} \Big(\frac{1}{n-p_+} \int_\rn f^{p_*} |\log f|  |\nabla \phii | |\nabla p|\,  dx + \frac{p_+}{n} \int_\rn f^{p_*-1} |\nabla f| |\nabla \phii|\, dx \Big).
\end{align*}

We need to estimate the two integrals on the right hand side of the previous inequality. Applying H\"older's inequality for variable exponent spaces for both integrals, we obtain
$$
\int_\rn f^{p_*} |\log f|  |\nabla \phii| |\nabla p|\, dx \leq 2 \| f \log f |\nabla p| \|_{p} \| f^{p_*-1} |\nabla\phii| \|_{p'}
$$
and
$$
\int_\rn f^{p_*-1} |\nabla f| |\nabla \phii|\, dx \leq 2 \| \nabla f \|_{p} \| f^{p_*-1} |\nabla\phii| \|_{p'}.
$$

\begin{remark}\label{p_->1}
In this point of the argument, we require that $p_->1$. Nevertheless, since the constants entering in the estimates do not depend on $p_-$ as we will see, it can be easily deduced that the result still holds for $p_-=1$.
\end{remark}

The modular giving the norm $ \| f^{p_*-1} |\nabla\phii| \|_{p'}$ can be expressed as
$$
\int_\rn \left(f^{p_*-1} |\nabla \phii| \right)^{p'}\, dx = \int_\rn |\nabla \phii|^{p'} F \,dx = \int_\rn |y|^{(p'\circ \nabla \phii^{-1})} G \, dy,
$$
where we have used the transport identity \eqref{transport.identity}. This term is completely analogous to the one appearing in \cite{CE-N-V} except that the exponent is not constant and therefore depends on $\phii$. So we need to bound it independently of $\phii$.

Let us consider the function $\eta(x) = \max \{ |x|^{(p')_+}, |x|^{(p')_-} \}$, where
$$(p')_+ = \sup \frac{p}{p-1} = (p_{-})'.$$
Hence
$$
\int_\rn \left(f^{p_*-1} |\nabla \phii| \right)^{p'}\, dx \le \int_\rn \eta G \, dy.
$$

Recalling the relation between the modular and the Luxemburg norm, we prove the following result

\begin{teo}\label{teo:est_transporte}
Let $f\in C^\infty_c(\rn)$ nonnegative and let $p$ be a smooth exponent that verifies \eqref{p-p+}. Assume that $\|f\|_{p^*}=1$, define $\eta(x) = \max \{ |x|^{(p')_+}, |x|^{(p')_-} \}$ and let
\begin{equation}\label{eq:alpha}
\alpha(n,p) = \frac{n-p_+}{2 (n-1)}\sup \frac{\int_\rn g^{p_*}\, dx}{\max\left\{\left(\int_\rn \eta g^{p^*} \, dy \right)^{\frac{1}{(p')_-}}, \left(\int_\rn \eta g^{p^*} \, dy \right)^{\frac{1}{(p')_+}}\right\}},
\end{equation}
where the supremum is taken over all nonnegative functions $g\in C^\infty_c(\rn)$ such that $\|g\|_{p^*}=1$.

Then, the following inequality holds:
\begin{equation}
\alpha(n,p) \leq \frac{1}{n-p_+} \| f \log f |\nabla p| \|_{p} +  \frac{p_+}{n} \| \nabla f \|_{p}.
\label{eq:est_transporte}
\end{equation}
\end{teo}

This last estimate is exactly the one obtained in \cite{CE-N-V} with the exception of the logarithmic term. Observe that if the logarithmic term is removed from \eqref{eq:est_transporte} then the Sobolev-Poincar\'e inequality \eqref{Sobolev-Poincare} is proved by homogeneity.

\begin{remark}\label{cota.alpha}
It is immediate to see that the term $\alpha(n,p)$ can be bounded below by a term depending only on $n$, $p_+$ and $p_-$. Moreover, since the constant does not degenerate when $p_-=1$ it can be taken depending only con $n$ and $p_+$.
\end{remark}

So, the remaining of the proof will be to bound the logarithmic term by a constant times some norm of $|\nabla p|$.

We begin with a lemma.
\begin{lema} \label{loglema}Let $f\in C^\infty_c(\R^n)$ be nonnegative, and $p$ be a smooth exponent such that \eqref{p-p+} holds. Assume $\|f\|_{p^*} = 1$ and take $s>n$. Then,
\begin{equation}
\int_\rn f^p |\log f|^p |\nabla p|^p\, dx \le  (C_1 +C_2) \max\{\|\nabla p\|_s^{p_-}, \|\nabla p\|_s^{p_+}\},
\label{eq:logcontrol}
\end{equation}
where
\begin{equation} \label{eq:defctes} \begin{split}
& C_1 = 2 \max \Bigg\{\left(C(n,s) \diam(\supp f) \|1_{\supp f}\|_{p'} \|\nabla f\|_p\right)^{\frac{1}{\left(\frac{s}{s-p}\right)_-}}; \\
& \hspace{2.4cm} \left(C(n,s) \diam(\supp f) \|1_{\supp f}\|_{p'} \|\nabla f\|_p\right)^{\frac{1}{\left(\frac{s}{s-p}\right)_+}}\Bigg\}, \\
& C_2 = \frac{2s(n-1)}{e(s-n)}.
\end{split} \end{equation}
%are given by
%$$
%C_1 = 4^{\frac{1}{q}} \left(\frac{qp_-}{e(qp_- -1)}\right)^{p_-},
%$$
%$$
%C_2 = \max \left\{ \left(\frac{q (n-p_-)}{n - q(n-p_-)} \right)^{p_-}, \left(\frac{q (n-p_-)}{n - q(n-p_-)} \right)^{p_+}\right\}.
%$$
\end{lema}

\begin{proof}
The proof is rather elementary. Let us split the integral into two parts, one where $f\le 1$ and the other one where $f>1$. First,  we have
\begin{equation*}
\int_{\{f \leq 1 \} } f^p |\log f|^p |\nabla p|^p\, dx \le
2\|f^p |\log f|^p 1_{\{f \leq 1 \} } \|_{\frac{s}{s-p}}\| |\nabla p|^p\|_{\frac{s}{p}}.
\end{equation*}

Now we need to bound the first norm. Let $r:=\frac{sp}{s-p}$ and observe that $r_->1$. Observe that if $f\le 1$ one has the control $f |\log f|\le e^{-1}<1$ and so
$$
f^r |\log f|^r \le f^{r_-} |\log f|^{r_-} = f (f^{r_- -1} |\log f|^{r_-}) \le C(r_-) f,
$$
where
$$
C(r_-) = \left(\frac{r_-}{e(r_- -1)}\right)^{r_-}\le C(n,s).
$$
Hence, by Poincar\'e inequality in $L^1$ and H\"older inequality for variable exponents,
\begin{equation}\label{f<1}
\begin{split}
\int_{\{f\le 1\}} f^r |\log f|^r\, dx &\le C(n,s) \int_\rn f\\
&\le C(n,s)2\diam(\supp f)\int_\rn |\nabla f|\, dx\\
&\le C(n,s) 4\diam(\supp f) \|1_{\supp f}\|_{p'} \|\nabla f\|_p
\end{split}
\end{equation}
On the other hand, it holds
$$
\||\nabla p|^p\|_{\frac{s}{p}}\le \max\{\|\nabla p\|_s^{p_-}, \|\nabla p\|_s^{p_+}\}
$$
and therefore
\begin{align*}
\int_{\{f \leq 1 \} } f^p |\log f|^p |\nabla p|^p\, dx \leq & C_1 \max\{\|\nabla p\|_s^{p_-}, \|\nabla p\|_s^{p_+}\} ,
\end{align*}
where $C_1$ is defined according to \eqref{eq:defctes}.

Now, for $f\ge 1$ we write
$$
f^p |\log f|^p = f^r \left( f^{1-\frac{r}{p}} |\log f| \right)^p.
$$
Then, if $r>p$ it is immediate to check that,
$$
f^{1-\frac{r}{p}} |\log f| \le f^{1-(\frac{r}{p})_-} |\log f| \le  \frac{1}{e((\frac{r}{p})_- - 1)} =: K(r,p).
$$

In consequence, we obtain
\begin{equation}\label{f>1}
\begin{split}
\int_{\{f>1\}} f^p |\log f|^p |\nabla p|^p\, dx &\leq K(r,p) \int_\rn f^r |\nabla p|^p\, dx\\
&\le 2 K(r,p) \|f^r\|_{\frac{s}{s-p}} \||\nabla p|^p\|_{\frac{s}{p}}.
\end{split}
\end{equation}

and if we take $r=p^*(1-\frac{p}{s})$, which verifies $r>p$ since $s>n$, it follows that $\frac{rs}{s-p} = p^*$ and so
$$
\|f^r\|_{\frac{s}{s-p}} = 1.
$$
Observe that by our choice of $r$ we actually have that
$$
K(r,p) \le C_2.
$$

Therefore, from \eqref{f>1} we obtain
\begin{equation}\label{f>1_2}
\int_{\{f>1\}} f^p |\log f|^p |\nabla p|^p\, dx \le  C_2 \max\{\|\nabla p\|_s^{p_-}, \|\nabla p\|_s^{p_+}\}.
\end{equation}
Putting together \eqref{f<1} and \eqref{f>1_2} we conclude the desired result.
\end{proof}

%\begin{remark}
%Since $q$ can be bounded above and below depending only con $n$, the constants in Lemma \ref{loglema} can be taken depending only on $n$ and $p_+$.
%\end{remark}

\begin{remark}An immediate consequence of the previous lemma is that the Sobolev inequality holds if $p$ satisfies the condition
\begin{equation}
\max\{\|\nabla p\|_s^{p_-}, \|\nabla p\|_s^{p_+}\}< \delta ,
\label{eq:hipp}
\end{equation}
for some $\delta > 0$ small enough.

In fact, \eqref{eq:est_transporte} and \eqref{eq:logcontrol} will give
\begin{equation}\label{eq:casisobolev}
\begin{split}
\alpha(n,p) - \frac{1}{n-p_+}\left(C_2\delta \right)^{\frac{1}{p+}} \le C_1^{\frac{1}{p+}}\max\{\|\nabla p\|_s^{p_-}, \|\nabla p\|_s^{p_+}\}^{\frac{1}{p+}}+ \frac{p_+}{n} \|\nabla f\|_p.
\end{split}
\end{equation}
So, by Remark \ref{cota.alpha}, we can choose $\delta=\delta(n,p_+)$ such that, if $\supp f\subset B_R$,
$$
0<D(n, p_+) \leq \left(B(n,p_+,R,s)\max \left\{\|\nabla p\|_s^{\frac{p_-}{p_+}}, \|\nabla p\|_s\right\} + C(n,p_+)\right)\| \nabla f \|_p,
$$
and the general inequality for all $f \in \cinfc(\rn)$ follows by homogeneity.
\end{remark}

Condition \eqref{eq:hipp} is by no means restrictive; indeed, it is possible to obtain the inequality for any exponent by means of a scaling argument. Given any measurable function $f$ on $\rn$ and $k\geq 1$ consider the function $f_k(x) = f \left( \frac{x}{k} \right)$. Then, the following lemma holds

\begin{lema} \label{scaling} Given $f\in C^\infty_c(\rn)$ nonnegative, $p$ a smooth exponent satisfying \eqref{p-p+} and $r\ge 1$ constant, we have
\begin{equation}
k^{\frac{n}{p^*_+}} \| f \|_{p^*} \leq \| f_k \|_{p_k^*} \leq k^\frac{n}{p^*_-} \| f \|_{p^*} \label{eq:festrella}
\end{equation}
\begin{equation}
k^{\frac{n}{p_+}-1} \| \nabla f \|_p \leq \| \nabla f_k \|_{p_k} \leq k^{\frac{n}{p_-} - 1} \| \nabla f \|_p \label{eq:gradf}
\end{equation}
\begin{equation}
k^{-p_+ + \frac{n}{r}} \| |\nabla p|^p \|_{r} \leq \| |\nabla p_k|^{p_k} \|_{r} \leq k^{-p_- + \frac{n}{r}} \| |\nabla p|^p \|_{r} \label{eq:gradp}
\end{equation}
\end{lema}

\begin{remark}
Since $p\mapsto p^*$ is monotone with respect to $p$, it follows that $p^*_- = (p_-)^*$ and therefore, by \eqref{p-p+}, that $p^*_-\ge 1^* = \frac{n}{n-1}$. So in Lemma \ref{scaling} the following upper bounds can be obtained
\begin{align}
\label{eq:festrella2}
&\|f_k\|_{p_k^*}\le k^{n-1} \|f\|_{p^*}\\
\label{eq:gradf2}
&\|\nabla f\|_{p_k}\le k^{n-1} \|\nabla f\|_p\\
\label{eq:gradp2}
&\| |\nabla p_k|^{p_k}\|_s\le k^{-1+\frac{n}{s}} \| |\nabla p|^p\|_s
\end{align}
\end{remark}

\begin{proof}
We will prove inequalities \eqref{eq:gradf}. Inequalities \eqref{eq:festrella} follow in the same way. Take $\mu >0$, then
\begin{align*}
\int_\rn \left(\frac{|\nabla f_k|}{\mu}\right)^{p_k}\, dx&= \int_\rn \left(\frac{|\frac{1}{k}\nabla f(\frac{x}{k})|}{\mu}\right)^{p(\tfrac{x}{k})}\, dx\\
&=\int_\rn \left(\frac{|\frac{1}{k}\nabla f|}{\mu}\right)^{p} k^n\, dy\\
&=\int_\rn\left(\frac{|\nabla f|}{\mu k^{1-\frac{n}{p}}}\right)^{p}\,dy
\end{align*}
Recall that $k\geq1$, so that
$$
\int_\rn \left(\frac{|\nabla f|}{\mu k^{1-\frac{n}{p_+}}}\right)^{p}\,dy\leq\int_\rn \left(\frac{|\nabla f_k|}{\mu}\right)^{p_k}\, dx \leq \int_\rn \left(\frac{|\nabla f|}{\mu k^{1 -\frac{n}{p_-}}}\right)^{p}\, dy .
$$
Then, choosing $\mu=\| \nabla f \|_p k^{-1+\frac{n}{p_+}}$, we obtain that
$$
\|\nabla f_k\|_{p_k}\geq\|\nabla f\|_p k^{-1+\frac{n}{p_+}}.
$$
Analogously, it holds that
$$
\|\nabla f_k\|_{p_k} \leq \|\nabla p\|_p k^{-1+\frac{n}{p_-}}.
$$

Now, in order to prove \eqref{eq:gradp}, observe that
$$
\int_{\R^n} |\nabla p_k|^{p_k r}\, dx = \int_{\R^n} |\nabla p|^{pr} k^{n-pr}\,dy.
$$
As before,
$$
\int_{\R^n} |\nabla p|^{pr} k^{n-p_+r}\,dy \leq \int_{\R^n} |\nabla p_k|^{p_k r}\,dx\leq\int_{\R^n}|\nabla p|^{pr} k^{n-p_-r}\,dy
$$
Inequality \eqref{eq:gradp} follows easily.
\end{proof}

\begin{remark}
As observed before, the constant $\alpha(n,p)$ can be bounded below by a constant depending only on $n$ and  $p_+$. Therefore, since $(p_k)_+ = p_+$, it follows that $\alpha(n,p_k)$ is bounded below independently of $k$.
\end{remark}

\begin{remark}\label{eq:delta} Take $\delta$ in \eqref{eq:casisobolev} so that the left hand side is $\frac{\alpha(n,p)}{2}$. Since $(p_k)_+ = p_+$, from \eqref{eq:casisobolev} we have\begin{equation}
\frac{\alpha(n,p_k)}{2} \leq C_{1,k}^\frac{1}{p_+} \max\{\||\nabla p_k\|_s^\frac{p_-}{p_+}, \|\nabla p_k\|_s\} + \frac{p_+}{n} \frac{\| \nabla f_k \|_{p_k}}{\| f_k \|_{p_k^*}},
\label{eq:sobolevconk}
\end{equation}
\begin{align*}
C_{1,k}=2 \max \Bigg\{ & \left(C(n,s) \diam(\supp f_k) \|1_{\supp f_k}\|_{p'} \frac{\|\nabla f_k\|_{p_k}}{\| f_k \|_{p_k^*}}\right)^{\frac{1}{\left(\frac{s}{s-p_k}\right)_-}};\\
&\left(C(n,s) \diam(\supp f_k) \|1_{\supp f_k}\|_{p'} \frac{\|\nabla f_k\|_{p_k}}{\| f_k \|_{p_k^*}}\right)^{\frac{1}{\left(\frac{s}{s-p_k}\right)_+}}\Bigg\}
\end{align*}
for all $f_k \in \cinfc(\rn)$, and $k\in \N$ such that $p_k$ verifies condition \eqref{eq:hipp}.

However, since $s>n$ it follows that $-p_-+\frac{n}{s}<0$ and so, by \eqref{eq:gradp} there exists $k_0=k_0(\delta, n, s, \||\nabla p|^p\|_{s})$ such that $p_k$ verifies condition \eqref{eq:hipp} for every $k\ge k_0$.
\end{remark}

With all this preliminaries we are now in position to prove the main theorem of the section.

\begin{teo} \label{sobolev}
Assume that $\nabla p\in L^s(\R^n)$ for some  $s > n$.

Then, the following Sobolev-Poincar\'e inequality holds
\begin{equation}
\| f \|_{p^*} \leq C(n,p_+, R) \| \nabla p \|_{s}^{\gamma(n,p_+)} \max \left\{\|\nabla f \|_p^{\frac{1}{\left(\frac{s}{s-p}\right)_+}}, \|\nabla f \|_p^{\frac{1}{\left(\frac{s}{s-p}\right)_-}} \right\}^\frac{1}{p_+}+ C(n,p_+) \| \nabla f \|_p
\label{eq:sobolev}
\end{equation}
for any $f\in C^\infty_c(\rn)$ such that $\supp f \subset B_R$.
\end{teo}

\begin{proof}
%In first place, observe that
%$$ \min \{ \| \nabla p \|_{s}^{p_-}, \| \nabla p \|_{s}^{p_+} \} \leq \| |\nabla p|^p \|_{\frac{s}{p}} \leq \max \{ \| \nabla p \|_{s}^{p_-}, \| \nabla p \|_{s}^{p_+} \}.$$
%So, our hypothesis on $\nabla p$ imply that $\| |\nabla p|^p\|_{\frac{s}{p}}$ is finite.
We have already proved that the desired inequality holds if $\max \{ \| \nabla p \|_{s}^{p_-}, \| \nabla p \|_{s}^{p_+} \}$ is smaller than a fixed constant.

Now, by \eqref{eq:gradp} and the fact that $s > \frac{n}{p_-}$, for any $p$ under our hypothesis we can fix $k>1$ large enough so that the Sobolev inequality holds for the exponent $p_k$. More precisely, it is enough to consider
\begin{equation}
k = \max \left\{ 1,  \frac{ \max \{ \| \nabla p \|_{s}^{p_-}, \| \nabla p \|_{s}^{p_+} \} }{\delta}\right\},
\label{eq:defk}
\end{equation}
where $\delta$ is as in Remark \ref{eq:delta}, so that if $\max \{ \| \nabla p \|_{s}^{p_-}, \| \nabla p \|_{s}^{p_+} \} \leq \delta$, no scaling is needed. In this case, recall that the Sobolev inequality is just \eqref{eq:casisobolev}.

Let us now suppose that $\max \{ \| \nabla p \|_{s}^{p_-}, \| \nabla p \|_{s}^{p_+} \} > \delta$. Given $f \in \cinfc (\rn)$ nonnegative such that $\|f\|_{p^*}=1$, the application of \eqref{eq:festrella}, \eqref{eq:gradf} and the Sobolev inequality \eqref{eq:sobolevconk} for $p_k$ yield
\begin{equation*}\begin{split}
\frac{\alpha(n,p)}{2} \leq  C(n,p_+,s) & k^{\alpha(n,p_+, s)} \\
\max \Bigg\{ & \left(\diam(\supp f) \|1_{\supp f}\|_{p'}\| \nabla f \|_p \right)^{\frac{1}{\left(\frac{s}{s-p}\right)_+}}, \\
& \left(\diam(\supp f)  \|1_{\supp f}\|_{p'}\| \nabla f \|_p \right)^{\frac{1}{\left(\frac{s}{s-p}\right)_+}} \Bigg\}^{\frac{1}{p_+}}
\max \{ \| \nabla p \|_{s}^{p_-}, \| \nabla p \|_{s}^{p_+} \} \\
& + k^{\beta(n,p_+)} \frac{p_+}{n} \| \nabla f \|_p.
\end{split}\end{equation*}

Combining the previous inequality with the election of $k$ we made, we can easily conclude the theorem.
\end{proof}

%%%%%%%%%%%%%%%%%%%%%%%%%%%%%%%%%%%%%%%%%%%%%%%%%%%%%%%%%%%%%%%%%%%%%%%%%%%%%%%%%%%%%%%%%%%%

\section{Trace inequality}

In this section we show the flexibility of mass transportation methods in dealing with Sobolev-type inequalities for variable exponents, applying the same type of arguments to the Sobolev trace inequality \eqref{trace}.

This method in the constant exponent case was first employed by Nazaret in \cite{Nazaret} where, in addition, the author was able to compute the exact value of the optimal constant along with the extremals answering positively to a question raised by Escobar in \cite{Escobar}.

Our arguments in this section follow closely the ones in \cite{Nazaret} until some point where some new terms appear, due to the non constant nature of the exponent.

So, consider $f, g \in \cinfc (\rn)$ nonnegative be such that  such that $F:= f^{p^*}$ and $G:=g^{p^*}$ are probability densities in $\rnm$ and proceeding in the same way as in the previous section we obtain
$$
\int_\rnm g^{p_*}\, dx \leq \frac{1}{n} \int_\rnm  F^{1-\frac{1}{n}}  \Delta \phii\, dx,
$$
where $\Delta \phii$ stands for the distributional laplacian of the convex function $\phii$ such that $\nabla \phii$ transport the measure $d\mu = F\, dx\lfloor_{\rnm}$ into $d\nu = G\, dx\lfloor_{\rnm}$.

For technical reasons, it is convenient to subsitute $\phii$ by a $\psi = \phii- \e \cdot x$, where $\e=(-1,0,\dots, 0)$ and since both functions $\phii$ and $\psi$ have the same laplacian, after integrating by parts the following estimate holds
$$
\int_\rnm g^{p_*}\, dx \leq - \frac{1}{n} \int_\rnm \nabla \left(F^{1-\frac{1}{n}}\right) \cdot \nabla \psi \, dx - \frac{1}{n} \int_{\R^{n-1}} F^{1-\frac{1}{n}}\, dx' + \frac{1}{n} \int_{\R^{n-1}} F^{1-\frac{1}{n}}\nabla\phii \cdot \e\,  dx'.
$$
Since $\nabla \phii\in \rnm$, it follows that $\nabla\varphi\cdot \e\leq0$ on $\R^{n-1}$, so
$$
\int_\rnm g^{p_*}\, dx \leq - \frac{1}{n} \int_\rnm \nabla \left(F^{1-\frac{1}{n}}\right) \cdot \nabla \psi \, dx - \frac{1}{n} \int_{\R^{n-1}} F^{1-\frac{1}{n}} dx'.
$$
Up to know it is exactly the same as in Nazaret's paper \cite{Nazaret}. See that paper for the details.

Now is where the differences arise.

Proceeding as in the previous section we can estimate the first integral in the right hand side,
\begin{align*}
- \frac{1}{n} \int_\rnm \nabla \left(F^{1-\frac{1}{n}}\right) \cdot \nabla \psi  dx & \leq \Big(\big(C(n,p_+,s, R)\|\nabla f\|_{p,\rnm}^{\alpha(n,p_+,s)} + C(n, p_+,s)\big) \max\{\|\nabla p\|_s^{p_-}, \|\nabla p\|_s^{p_+}\}\\
& + C(n,p_+) \| \nabla f \|_{p,\rnm} \Big) \max\left\{\left(\int_\rn \tilde{\eta} G \, dy \right)^{\frac{1}{(p')_-}}, \left(\int_\rn \tilde{\eta} G \, dy \right)^{\frac{1}{(p')_+}}\right\},
\end{align*}
where $\tilde{\eta}(y) = \eta(y-\e) =  \max \{ |y-\e|^{(p')_+}, |y-\e|^{(p')_-} \}$.

Then, we reach the key estimate
\begin{equation} \begin{split}
\int_{\R^{n-1}} f^{p_*}\, dx' & \leq
 \Big(\big(C(n,p_+,s, R)\|\nabla f\|_{p,\rnm}^{\alpha(n,p_+,s)} + C(n, p_+,s)\big) \max\{\|\nabla p\|_s^{p_-}, \|\nabla p\|_s^{p_+}\}\\
& + C(n,p_+) \| \nabla f \|_{p,\rnm} \Big) \max\left\{\left(\int_\rn \tilde{\eta} G \, dy \right)^{\frac{1}{(p')_-}} \hspace{-0.1cm}, \left(\int_\rn \tilde{\eta} G \, dy \right)^{\frac{1}{(p')_+}}\right\} -  n \int_\rnm g^{p_*}\, dx,
\label{eq:trace}
\end{split} \end{equation}
which is valid for all $f, g\in C^\infty_c(\rn)$ nonnegative such that $\|f\|_{p^*, \rnm}=\|g\|_{p^*, \rnm}=1$.

Now, we denote by
$$
\beta(n,p) = \sup \frac{n\int_\rnm g^{p_*}\, dx}{\max\left\{\left(\int_\rn \tilde{\eta} G \, dy \right)^{\frac{1}{(p')_-}}, \left(\int_\rn \tilde{\eta} G \, dy \right)^{\frac{1}{(p')_+}}\right\}}
$$
where the supremum is taken over all nonnegative $g\in C^\infty_c(\rn)$ such that $\|g\|_{p^*,\rnm}=1$.

So, if we assume that $\max\{\|\nabla p\|_s^{p_-}, \|\nabla p\|_s^{p_+}\}< \delta$ for some $\delta<\delta(n,p_+,s)$, namely $C(n,p_+,s)\delta<\frac12 \beta(n,p)$, then
$$
C(n,p_+,s) \max\{\|\nabla p\|_s^{p_-}, \|\nabla p\|_s^{p_+}\} \max\left\{\left(\int_\rn \tilde{\eta} G \, dy \right)^{\frac{1}{(p')_-}}, \left(\int_\rn \tilde{\eta} G \, dy \right)^{\frac{1}{(p')_+}}\right\} - \frac{n}{2}\int_\rnm g^{p_*}\, dx < 0,
$$
for some $g\in C_c^\infty(\rn)$ nonnegative with $\|g\|_{p^*,\rnm}=1$. Observe that at this point we require $p_->1$. Therefore, for such $g$ we have
\begin{equation}\label{casi.traza}
\int_{\R^{n-1}} f^{p_*}\, dx' \leq C(n,p_+,s, R)\max\{\|\nabla p\|_s^{p_-}, \|\nabla p\|_s^{p_+}\} \| \nabla f \|_{p,\rnm}^{\alpha(n,p_+,s)}+ C(n,p_+,s) \| \nabla f \|_{p,\rnm} - B(n,p),
\end{equation}
where $B(n,p)>0$.

The argument for a general exponent will be just as in Theorem \ref{sobolev}, so without loss of generality we can assume that \eqref{casi.traza} holds.

Now, the Sobolev trace inequality \eqref{trace} follows from \eqref{casi.traza} using the same scaling arguments as in \cite{Nazaret}. We sketch these arguments for the reader's convenience.

Assume first that $\|f\|_{p_*,\R^{n-1}}\ge 1$, the other case is completely analogous. Then, from equation \eqref{casi.traza}, it easily follows that
$$
Q(f) \left(\frac{\|\nabla f\|_p}{\|f\|_{p^*}}\right)^{(p_*)_+}\leq C_1 \left(\frac{\|\nabla f \|_p}{\| f \|_{p^*}}\right)^\alpha + C_2 \frac{\|\nabla f \|_p}{\| f \|_{p^*}} - B,
$$
where $\supp f\subset B_R$, $C_1 = C_1(n,p_+,s,R)$, $C_2 = C_2(n,p_+,s)$, $\alpha=\alpha(n,p_+,s)$, $B=B(n,p_+,p_-)$ and
$$
Q(f)=\frac{\int_{\R^{n-1}}f^{p_*} \, dx'}{\|\nabla f \|_p^{(p_*)_+}}.
$$

Equivalently, considering $t=\frac{\|\nabla f\|_p}{\|f\|_{p^*}}$,
$$
Q(f) \leq  \left( C_1 t^\alpha + C_2 t - B\right) t^{-(p_*)_+} =: h(t).
$$

Now, it is easy to see that $h(t)$ is bounded above for $t>0$ for some constant depending on $C_1, C_2, B, \alpha$ and $p_+$ which in turn depends on $n,p_+,s,  \| \nabla p\|_{s, \rnm}$ and $R$.

We conclude that there exists some $C(n,p_+,s,\|\nabla p\|_{s,\rnm}, R)>0$ such that
$$ \int_{\R^{n-1}}f^{p_*} \, dx' \leq C  \|\nabla f \|_{p, \rnm}^{(p_*)_+} $$
for all $f \in \cinfc (\rn)$ satisfying $\|f\|_{p^*, \rnm}\geq 1$.

Analogously, if $f \in \cinfc(\rn)$ is such that $\|f\|_{p^*, \rnm}\leq 1$, it follows that
$$
\int_{\R^{n-1}} f^{p_*} \, dx' \leq C \|\nabla f \|_{p, \rnm}^{(p_*)_-} .
$$

Summing up, we have proved the following theorem.
\begin{teo}\label{teo:trace}
Let $p$ be an exponent satisfying $1<p_-\leq p_+ \leq n$ and $|\nabla p|\in L^s(\rnm)$ for some $s>n$. Then, there exists some constant $C=C(n,p_+,p_-,s, \|\nabla p\|_{s,\rnm}, R)$ such that
$$
\|f\|_{p_*,\R^{n-1}}\le C \|\nabla f\|_{p,\rnm},
$$
for every $f\in C^\infty_c(\rn)$ such that $\supp f\subset B_R$.
\end{teo}

\section*{Acknowledgments}

This paper was partially supported by Universidad de Buenos Aires under grant UBACyT 20020130100283BA, by CONICET under grant PIP 2009 845/10 and by ANPCyT under grant PICT 2012-0153. J. Fern\'andez Bonder and A. Silva are members of CONICET and J.P. Borthagaray is a doctoral fellow of CONICET.

\bibliographystyle{plain}
\bibliography{biblio}
\end{document}